\documentclass[12pt,reqno]{amsart}
\usepackage[english]{babel}
\usepackage[cp1251]{inputenc}

\usepackage{mathrsfs}
\usepackage{amsthm}
\usepackage{amssymb}
\usepackage{amsmath,amsthm,amsfonts}
\usepackage{mathabx}
\usepackage{hhline}
\usepackage{textcomp}
\usepackage{amsmath,amscd}
\usepackage[all,cmtip]{xy}
\usepackage{mathtools}
\usepackage[margin=1in]{geometry}
\usepackage{lipsum}
\usepackage{extarrows}
\usepackage{extarrows}
\author[]{Kaveh Eftekharinasab}
\title[]{Transversality and Lipschitz-Fredholm maps }
\address{Topology dept. \\ Institute of Mathematics of NAS of Ukraine \\ Te\-re\-shchen\-kivska st. 3, Kyiv, 01601 Ukraine}
\email{kaveh@imath.kiev.ua}

\makeatletter
\@addtoreset{equation}{section}
\@addtoreset{figure}{section}
\@addtoreset{table}{section}
\makeatother

\newtheorem{theorem}{Theorem}[section]
\newtheorem{lemma}{Lemma}[section]
\newtheorem{remk}{Remark}[section]
\newtheorem{prop}{Proposition}[section]
\newtheorem{defn}{Definition}[section]
\newtheorem*{theorem*}{Theorem}
\newtheorem{cor}{Corollary}[section]
\theoremstyle{definition}

\DeclareMathAlphabet{\mathpzc}{OT1}{pzc}{m}{it}

\newcommand{\rr}{\mathbb{R}}
\newcommand{\nn}{\mathbb{N}}

\newcommand{\ab}{\mathcal{A}}
\newcommand{\mc}{MC^k}

\newcommand{\tn}{\pitchfork}
\DeclareMathOperator{\dd}{d}
\DeclareMathOperator{\codim}{codim}
\DeclareMathOperator{\Ind}{Ind}
\DeclareMathOperator{\Img}{Img}

\DeclareMathOperator{\Aut}{Aut}
\DeclareMathOperator{\Iso}{Iso}

\begin{document}

\begin{abstract}
We study  transversality for Lipschitz-Fredholm maps in the context of
bounded Fr\'{e}chet manifolds. We show that the set of all Lipschitz-Fredholm maps of a fixed index between  Fr\'{e}chet spaces  has
the transverse stability property.
We give a straightforward extension of the Smale transversality theorem by using the generalized Sard's theorem for this category
of manifolds.
We also provide an answer to the well known problem concerning the existence of a submanifold structure on the preimage of a transversal submanifold. 
\end{abstract}
\maketitle
\section{Introduction}
In~\cite{k} we proved a version of the classical Sard-Smale theorem for a category of generalized Fr\'{e}chet manifolds,
bounded (or $\mc$) Fr\'{e}chet manifolds, introduced in~\cite{m}. Our approach to the theorem's generalization  is based on the
assumption that Fredholm operators need to be globally Lipschitz. A reason for this interest is that there exists an appropriate
topology on $\mathcal{L}(E,F)$, the space of all linear globally Lipschitz maps between Fr\'{e}chet spaces $E$ and $F$, that leads to the openness of
the set of linear isomorphisms in $\mathcal{L}(E,F)$~\cite[Proposition 2.2]{k}. This result in turn yields
the openness of the collection of Fredholm operators in $\mathcal{L}(E,F)$~\cite[Theorem 3.2]{k}. The other reason is
that Lipschitzness  is consistent with the notion of differentiability, bounded (or $\mc$-) differentiability, that we apply. 
If $E,F$ are Fr\'{e}chet spaces and if $U$ is an open subset of $E$, a map $f :  U \rightarrow F$
is called bounded (or $MC^1$-) differentiable if it is Keller-differentiable, the directional derivative $\dd f (p)$ belongs to $ \mathcal{L}(E,F)$ for all
$p \in U$, and the induced map
$\dd f : U \rightarrow \mathcal{L}(E,F)$ is continuous.
Thus,  we can naturally define the index of a  Fredholm map between manifolds.

We should point out that  the mentioned results stems from the essential fact  that under a certain condition we can endow the space $\mathcal{L}(E,F)$
with a topological group structure. Also, the group of automorphisms of a Fr\'{e}chet space $E$, $\Aut(E)$, is
open in $\mathcal{L}(E,E)$ \cite[Proposition 2.1]{glockner}. But, in general, the group
of automorphisms of a Fr\'{e}chet space does not admit a non-trivial topological group structure. Thus, without some
restrictions it would be impossible to establish openness of sets of linear isomorphisms and Fredholm operators.
This is a major obstruction in developing the Fredholm theory for Fr\'{e}chet spaces.

A crucial step in the proof of an infinite dimensional version of Sard's theorem is that,
roughly speaking, for a Fredholm map $f:M\rightarrow N$ of manifolds, at each point $p \in M$, we may find local charts $(p \in U \subseteq M,\phi)$ and $(f(p) \in V \subseteq N,\psi)$  
 such that in the charts $f$ has a representation of the form
$f(u,v)=(u,\eta(u,v))$, where $\eta :\phi(U) \rightarrow \rr^n$ is some smooth map. This is a consequence
of an inverse function theorem. One of the main significance of the category of bounded Fr\'{e}chet manifolds
is the availability of an inverse function theorem in the sense of Nash and Moser~\cite[Theorrem 4.7]{m}.
However, the bounded differentiability is strong and in some
cases the class of bounded maps can be quite small, e.g. when the identity component
of $\mathcal{L}(E,F)$ contains only the zero map~\cite[Remark 2.16]{glockner}.

We have argued that why we have utilized this particular category of Fr\'{e}chet manifolds.
A salient example of these manifolds is the space of all smooth sections of a fiber bundle
over closed or non-compact manifolds (\cite[Theorem 3.34]{m}). On the other hand, it turns out that these generalized manifolds 
can surpass the geometry of Fr\'{e}chet manifolds. On these manifold
we are able to give a precise analytic meaning to some essential geometric objects (such as connection maps, vector fields and integral curves)
~\cite{k1}. Therefore, we would expect their applications to problems in global analysis.

The present work studies the differential topology of Lipschitz-Fredholm maps in the bounded Fr\'{e}chet setting. 
We show that the set of Lipschitz-Fredholm operators
of index $l$ between Fr\'{e}chet spaces $E$ and $F$ is open in the space of linear globally Lipschitz maps endowed with the fine topology.
We say that a set of maps has the transverse stability property for the fine topology if
maps in a fine neighborhood of a given map have the same transversality property i.e. if $f:E\rightarrow F$ is a 
map transversal to a closed subspace $\mathbb{F}$ of $F$, then any  map in a fine neighborhood of $f$ is
transversal to $\mathbb{F}$.  We then prove that the set of all Lipschitz-Fredholm maps of a fixed index between  Fr\'{e}chet spaces  has
the transverse stability property. We also study transversality for Lipschitz-Fredholm maps on manifolds. 
We give a straightforward generalization of the Smale transversality theorem (\cite[Theorem 3.1]{smale}) by using our generalized Sard's theorem. 
Finally, we prove that if $f:M \rightarrow N$ is an $\mc$ Lipschitz-Fredholm map of manifolds
which is transversal to a finite dimensional submanifold $\ab$ of $N$, then $f^{-1}(\ab)$ is a submanifold of $M$.

We stress that these results can not be proved without strong restrictions. However,
the basic concepts of infinite dimensional differential topology  such as submanifold and  transversality  
can be simply come over from the Banach setting.

Our motivation for the present work, in the light of~\cite{k1}, lay in the desire to develop transversality tools for 
the degree theory,
including the Leray-Schauder degree, for Lipschitz-Fredholm maps, to derive applications to the study of solutions 
to systems of nonlinear partial differential or integral equations on spaces of smooth sections which are not linear.

\section{Preliminaries}
We shall recall the required definitions from the category of $\mc$ manifolds briefly
but in a self-contained way for the convenience of the reader, which also give us the opportunity to establish our
notations for the rest of the paper. For more studies we refer to~\cite{k,m,k1}.

Let $(F,d)$ be a Fr\'{e}chet space whose topology is defined by a complete translational-invariant metric $d$. 
A metric with absolutely convex balls will be called a standard metric. 
Every Fr\'{e}chet space  admits a standard metric which defines its topology. We shall always define the
topology of Fr\'{e}chet spaces with this type of metrics.

Let $(E,g)$ and $(F,d)$ be Fr\'{e}chet spaces and let $\mathcal{L}_{g,d}(E,F)$ be the set of all 
linear maps $ L: E \rightarrow F $ such that 
\begin{equation*}
\mathpzc{Lip} (L )_{g,d}\, \coloneq \displaystyle \sup_{x \in E\setminus\{0\}} \dfrac{d (L(x),0)}{g( x,0)} < \infty.
\end{equation*}
The transversal-invariant metric 
\begin{equation} \label{metric}  
 D_{g,d}: \mathcal{L}_{g,d}(E,F) \times \mathcal{L}_{g,d}(E,F) \longrightarrow [0,\infty) , \,\,
(L,H) \mapsto \mathpzc{Lip}(L -H)_{g,d} \,,
\end{equation}
on $\mathcal{L}_{d,g}(E,F)$ turns it  into an Abelian topological group (\cite[Remark 2.1]{k}). A map $ \varphi \in \mathcal{L}_{g,d}(E,F) $  
is called Lipschitz-Fredholm operator if its kernel has finite dimension and its image is closed and has finite co-dimension. 
The index of $ \varphi $, $\Ind \varphi$, is defined by 
$
\Ind \varphi = \dim \ker \varphi - \codim \Img \varphi.
$
We denote by $ \mathcal{LF}(E,F) $  the set of all Lipschitz-Fredholm operators, and by $ \mathcal{LF}_l(E,F) $ the subset of $ \mathcal{LF}(E,F) $  
consisting of those operators  of index $l$.
\begin{prop}\cite[Proposition 2.2]{k} \label{open}
 The set of linear isomorphisms from
$ E $ into $ F $, $ \Iso{(E,F)}, $ is open in $ \mathcal{L}_{d,g}(E,F) $ with respect to the topology induced by the Metric \eqref{metric}.
\end{prop}

\begin{theorem}[\cite{k}, Theorem 3.2]\label{th}
$ \mathcal{LF}(E,F) $ is open in $\mathcal{L}_{g,d}(E,F)$ with respect to the topology defined by the Metric
 \eqref{metric}. Furthermore, the function $ T \rightarrow \Ind T $ is continuous on $ \mathcal{LF}(E,F) $,
 hence constant on connected components of $ \mathcal{LF}(E,F) $.
\end{theorem}
A subset $ G $ of a Fr\'{e}chet space $ F $ is called topologically complemented (or it splits in $ F $), if $ F $
is homeomorphic to the topological direct sum $ G \oplus H $, where $ H $ is a subspace of $ F $. 
We call $ H $ a topological complement of $ G $ in $ F $.
\begin{theorem}[\cite {m}, Theorem 3.14] \label{comp2}
Let $ E $ be a Fr\'{e}chet space. Then
\begin{enumerate}
\item Every finite-dimensional subspace of $ E $ is closed.
\item Every closed subspace $ G \subset E $ with $ \codim(G)=\dim (E/G)< \infty $ is topologically complemented in $ E $.
\item Every finite-dimensional subspace of $ E $ is topologically complemented.
\item A linear subspace $G$ of $E$ has a topological complement $H$ if and only if there exists a 
continuous projection $\mathrm{Pr}$ of $E$ onto $H$, see~\cite{ko}.
\end{enumerate}
\end{theorem}
Let $ E,F $ be Fr\'{e}chet spaces, $ U $ an open subset of $ E $, and $ P:U \rightarrow F $
 a continuous map. Let $CL(E,F)$ be the space of all continuous linear maps from $E$ to $F$ topologized by the compact-open topology. 
 We say $ P $ is differentiable at the point $ p \in U$ if the directional derivative  
$\operatorname{d}P(p)$ exists in all directions $ h \in E $. 
If $P$ is differentiable at all points $p \in U$, if $\operatorname{d}P(p) : U \rightarrow CL(E,F)$ is continuous for all
$p \in U$ and if the induced map $ P': U \times E \rightarrow F,\,(u,h) \mapsto \operatorname{d}P(u)h $
 is continuous in the product topology, then we say that $ P $ is Keller-differentiable.
 We define $ P^{(k+1)}: U \times E^{k+1} \rightarrow F $ in the obvious inductive fashion.

If $P$ is Keller-differentiable, $ \operatorname{d}P(p) \in \mathcal{L}_{d,g}(E,F) $ for all $ p \in U $, and the induced map 
$ \operatorname{d}P(p) : U \rightarrow \mathcal{L}_{d,g}(E,F)   $ is continuous, then $ P $ is called bounded differentiable. We say $ P $ is $ MC^{0} $ 
and write $ P^0 = P $ if it is continuous. 
We say $P$ is an $ MC^{1} $ and write  $P^{(1)} = P' $ if it is bounded differentiable. Let $ \mathcal{L}_{d,g}(E,F)_0 $ be 
the connected component of $ \mathcal{L}_{d,g}(E,F) $ containing the zero map. If $ P $ is bounded differentiable and if 
 $V \subseteq U$ is a connected open neighborhood of $x_0 \in U$, then $P'(V)$ is connected and hence contained in the connected component
$P'(x_0) +  \mathcal{L}_{d,g}(E,F)_0 $ of $P'(x_0)$ in $\mathcal{L}_{d,g}(E,F)$. Thus, $P'\mid_V - P'(x_0):V \rightarrow \mathcal{L}_{d,g}(E,F)_0 $
 is again a map between subsets of Fr\'{e}chet spaces. This enables a recursive definition: If $P$ is $MC^1$ and $V$ can be chosen for each
 $x_0 \in U$ such that $P'\mid_V - P'(x_0):V \rightarrow \mathcal{L}_{d,g}(E,F)_0 $ is $ MC^{k-1} $, 
then $ P $ is called an $ MC^k$-map. We make a piecewise definition of $P^{(k)}$ by $ P^{(k)}\mid_V \coloneq \left(P'\mid_V - P'(x_0)\right)^{(k-1)} $
for $x_0$ and $V$ as before. 
The map $ P $ is $ MC^{\infty} $ (or smooth) if it is $ MC^k $ for all $ k \in \mathbb{N}_0 $. We shall denote the derivative of $P$ at $p$ by $\operatorname{D}P(p)$.
Note that $\mc$-differentiability implies the usual $C^k$-differentiability for maps of finite dimensional manifolds. 

Within this framework we can define $\mc$ Fr\'{e}chet manifolds, $\mc$-maps of manifolds
and tangent bundle over $\mc$ manifolds in obvious fashion way. We assume that manifolds are connected and second countable.

Let $f: M \rightarrow N$ $(k\geqq1)$ be an $\mc$-map of manifolds. We denote by $T_xf: T_xM \rightarrow T_{f(x)}N$
the tangent map of $f$ at $x \in M$ from the tangent space $T_xM$ to the tangent space $T_{f(x)}N$.
We say that $f$ is an immersion (resp. submersion)
provided $T_xf$ is injective (resp. surjective) and the range $\Img(T_xf)$ (resp. the kernel $\ker(T_xf) $)
splits in $T_{f(x)}N$ (resp. $T_xM$) for any $x \in M$. 
An injective immersion $f: M \rightarrow N$ which gives an isomorphism onto a submanifold
of $N$ is called an embedding.
A point $ x \in M $ is called a regular point  if 
$ \operatorname{D}f(x): T_xM  \longrightarrow T_{f(x)}N $ is surjective. The corresponding
value $f(x)$ is a regular value. Points and values other than regular are called critical points and values, respectively.

Let $ M $ and $ N $ be $\mc$ manifolds, $k\geqq 1$. 
A Lipschitz-Fredholm map is an $MC^1$-map $ f:M \rightarrow N $ such that for each $ x \in M $
the derivative $ \operatorname{D}f(x): T_xM  \longrightarrow T_{f(x)}N$ is a Lipschitz-Fredholm operator. 
The index of $ f$, denoted by $\Ind{f}$, is defined to be the index of $ \operatorname{D}f(x) $ for some $ x $.
Since $ f $ is $\mc$ and $ M $ is connected in the light of Theorem \ref{th} the definition does not depend on the choice of $ x $.
\section{Transversality and openness}
Let $F_1$ be a linear closed subspace of a Fr\'{e}chet space $F$ that splits in $F$.
Given $\mc$ manifold $M$ modelled on $F$, a subset $M_1$ of $M$ is a submanifold of $M$ modelled on $F_1$
provided there is $\mc$-atlas $\{(U_i,\phi_i)\}_{i \in I}$ on $M$ that induces an atlas on $M_1$, i.e.
for any $i\in I$ there are open subsets $V_i,W_i$ of $F,F_1$ such that $\phi_i(U_i) = V_i \oplus W_i$
and $\phi_i(U_i \cap M_1) = V_i \oplus \{ 0\}$ is open in $F_1$. We say that $M_1$ is a submanifold
of Banach type if $F_1$ is a Banach space, and a submanifold of finite type if $F_1 = \rr^n$ for some $n \in \nn$.

Let $C(E,\rr^{+})$ be the set of all continuous functions from $E$ into $\rr^{+}$, $h \in \mathcal{L}_{g,d}(E,F)$ and  $\varepsilon \in C(E,\rr^{+})$. 
A map $f \in \mathcal{L}_{g,d}(E,F)$ is called a $\varepsilon$-approximation
to $h$ if $ d(f(x),h(x)) \leqq \varepsilon (x)$ for all $x \in E$, we write $d(f,h) < \varepsilon$ for short. 
If we take the $\varepsilon$-approximation to $h$ to be  a neighborhood of $h$ in the set $\mathcal{L}_{g,d}(E,F)$,
then we obtain a topology. This topology is called the fine topology and 
we denote the resulting space by $L_{fine}^{0}(E,F)$.

Let $M$ and $N$ be $\mc$ manifolds modelled on Fr\'{e}chet spaces $E$ and $F$, respectively.
Let $\mc(M,N),1\leqq k\leqq \infty,$ be the set of $\mc$-maps from $M$ into $N$. Two maps $f,h \in \mc(M,N)$ 
are said to be equivalent at $x \in M$ if $T_x^kf =T_x^kh$, where $T^k$ is the $k$-th tangent map.
We define the $k$-jet of $f$ at $x$, $j_x^kf$, to be the equivalence class of $f$. Let $d_k$ be a fiber metric
on the tangent space $T^k_xM$ that induces a Fr\'{e}chet topology which is isomorphic to $E$. We describe the fine
topology of order $k$ on $\mc(M,N)$ as follows. Let $\varphi \in \mc(M,N)$ and $\Omega\coloneq\{V_i\}_{i \in I}$ be a locally finite cover of $M$. 
Let $\epsilon _i :V_i\rightarrow \rr^{+}$ be continuous for all $i \in I$. 
Then, the sets
$$
\Theta (\varphi,V_i,\epsilon_i) \coloneq \{\phi \in \mc(M,N) \mid d_k (j^k_x\phi,j^k_x\varphi) < \epsilon_i(x), x \in V_i \}
$$
constitute a basis for fine open neighborhoods of $\varphi$. In this case we say that $\phi$ in a fine neighborhood
of $\varphi$ is an $\mc$ fine approximation to $\varphi$.
\begin{lemma}
 The fine topology is finer than the topology induced by the Metric~\eqref{metric}.
\end{lemma}
\begin{proof}
 We must show that if  $\mathcal{N}(f,\delta)$ is a $\delta$-neighborhood of $f$, then we can find $\epsilon>0$
 such that if $D_{g,d}(f,h)<\epsilon$, then $h \in \mathcal{N}(f,\delta) $. Given a map $h \in \mathcal{L}_{g,d}(E,F)$, let
 $$\epsilon \coloneq \min \{ 1, \displaystyle \inf_{x \in E\setminus\{0\}} \dfrac{\delta(x)}{g(h(x),0)} \}.$$
 Now suppose $D_{g,d}(f,h)<\epsilon$, then we can easily see that $d(f,h)<\delta$ and hence $h \in \mathcal{N}(f,\delta)$.
\end{proof}
\begin{remk}\label{open2}
 We know that $($Proposition~\ref{open}$)$ $\Iso(E,F)$ is open in  $\mathcal{L}_{g,d}(E,F)$ endowed with the topology induced by
 the metric~\eqref{metric}. By the preceding lemma the fine topology is finer than the metric topology, thereby, $\Iso(E,F)$
 is  open in $L_{fine}^{0}(E,F)$ .
\end{remk}

\begin{defn}
 Let $f:E \rightarrow F$ be a Lipschitz-Fredholm operator of Fr\'{e}chet spaces.
 We say that $f$ is transversal to a closed subset $F_0 \subseteq F$ and write $f \tn F_0$ if
 \begin{enumerate}
  \item $\Img(f) + F_0 = F$, and
  \item either $F_0$ splits in $F$ or $f^{-1}(F_0)$ splits in $E$.
 \end{enumerate}
\end{defn}
The following result characterizes the transversality of Lipschitz-Fredholm operators.
\begin{prop}\label{spl}
Let $\varphi \in \mathcal{LF}_l(E,F)$. Suppose $F_0 \subseteq F$ is a closed subset such that $\Img(\varphi)+F_0 = F$.
Then $\varphi \tn F_0$ if and only if there are closed subsets $F_1 \subseteq F$ and $E_0\subseteq E$ with $F = F_0 \oplus F_1$
and $E = E_0 \oplus (E_1 \coloneq \varphi^{-1}(F_1))$ such that $\varphi_1 \coloneq \varphi |_{E_1} \in \Iso (E_1,F_1)$.
\end{prop}
\begin{proof} Assume that such a closed subset $F_0$ is given and $\varphi \tn F_0$.
 $(\Img(\varphi) \cap F_0)$ splits in $F_0$ because $m = \dim (F_0/\Img(\varphi)) \leqq \dim (F/\Img(\varphi)) < \infty$ and hence
 by Theorem~\ref{comp2}(2) there exists a space $\mathbb{F} \subseteq F_0$ of dimension $m$ such that 
 $F_0 = (\Img(\varphi) \cap F_0) \oplus \mathbb{F}$. Since $\Img(\varphi) \cap \mathbb{F} \subseteq \Img(\varphi) \cap F_0 $ it
 follows that $\Img(\varphi) \cap \mathbb{F} = \{0\}$. Also,
 $\Img(\varphi) + \mathbb{F} = (\Img(\varphi)+ (\Img(\varphi) \cap F_0)) +\mathbb{F} = \Img(\varphi)+ F_0 =F$. Thus,
 $\Img(\varphi) \oplus \mathbb{F} =F$, therefore, $\codim \Img(\varphi) =m$ and $\dim \ker(\varphi) = l + m$. Moreover, there exists
 a closed subset $\mathbb{E} \subseteq E$ such that $E= \ker (\varphi) \oplus \mathbb{E}$. The operator 
 $\Phi \coloneq \varphi |_{\mathbb{E}} \in \mathcal{L}(\mathbb{E},\Img(\varphi))$ is injective onto $\Img(\varphi)$, hence,
 by virtue of open mapping theorem is a homeomorphism and therefore $\Phi \in \Iso (\mathbb{E},\Img(\varphi))$.
 Let $\mathbb{E}_0 \coloneq \Phi^{-1}(\Img(\varphi) \cap F_0) \subseteq \mathbb{E}$, then 
 $E_0 = \varphi^{-1}(\Img(\varphi) \cap F_0) = \ker(\varphi) \oplus \mathbb{E}_0$.
 
 $\mathbb{E}_0$ is complemented in $E_0$ so there is a continuous projection $\mathrm{Pr_1}$ of $E_0$ onto $\mathbb{E}_0$ 
 (see Theorem~\ref{comp2}(4)). If $E_0$ is complemented in $E$, then there exists a continuous projection $\mathrm{Pr_2}$ of $E$
 onto $E_0$. Thus, $\mathrm{Pr}_1 \circ \mathrm{Pr}_2$ is a continuous projection from $E$ to $\mathbb{E}_0$ and its
 restriction to $\mathbb{E}$ is a again continuous projection onto $\mathbb{E}_0$, thereby, $\mathbb{E}_0$ is complemented
 in $\mathbb{E}$. This means there is a closed  subset $E_1 \subseteq \mathbb{E}$ (which is also closed in $E$) such that $\mathbb{E} = E_1 \oplus \mathbb{E}_0$.
 
 By the same argument we have, if $F_0$ is complemented in $F$, then $(\Img(\varphi) \cap F_0)$ is complemented
 in $\Img(\varphi)$ because $(\Img(\varphi) \cap F_0)$ is complemented in $F_0$. This means there is a closed subspace
 $F_1 \subseteq \Img(\varphi)$ (which is also closed in F) such that $\Img(\varphi) = F_1 \oplus (\Img(\varphi) \cap F_0) $.
 Therefore, we have $E = \ker{(\varphi)} \oplus \mathbb{E}_0 \oplus E_1 = E_0\oplus E_1$ and 
 $F = (\Img(\varphi) \cap F_0) \oplus \mathbb{F} \oplus F_1 = F_0 \oplus F_1$ and $\varphi_1=\Phi|_{E_1} \in \Iso(E_1,F_1)$.
 Moreover, $E_1=\varphi_1^{-1}(F_1)$. The converse is obvious.
\end{proof}
\begin{prop}\label{zero}
$\mathcal{LF}_l(E,F)$ is open in $\mathcal{L}_{fine}^0(E,F)$.
\end{prop}
\begin{proof}
Let $\varphi \in \mathcal{LF}_l(E,F)$. We show that there exists $\varepsilon>0$
 such that any $\phi \in \mathcal{L}_{g,d}(E,F)$ which is $\varepsilon$-approximation to $\varphi$ is a
 Lipschitz-Fredholm operator of index $l$.

 First we prove for the case $l=0$, then we show that the general case can be reduced to the case $l=0$.
 Let $L:E\rightarrow F$ (called a corrector) be a linear globally Lipschitz map having finite dimensional range such that 
 $K\coloneq L+\varphi$ is an isomorphism.
 Such a linear map always exists. Indeed, $L$ can be any linear globally Lipschitz map from $E$ into $F$ such that 
 $\ker(L) \oplus \ker(\varphi) = E$ and $\Img(L) \oplus \Img(\varphi) =F$.
 Choose $\varepsilon \in (0,1/ 2\mathpzc{Lip}(K^{-1}))$ small enough 
 and suppose that $\phi , \mathbb{L} \in \mathcal{L}(E,F)$ are $\varepsilon$-approximation to $\varphi$ and the dimension of the image of
 $\mathbb{L}$ is finite. Then
 $\mathbb{K}= \mathbb{L}+\phi$ satisfies $d (K(x),\mathbb{K}(x))< 1/ \mathpzc{Lip}(K^{-1})$, for all $x \in E$, thus 
 $\mathbb{K}$ is an isomorphism (see Remark~\ref{open2}) and hence $\phi \in \mathcal{LF}(E,F)$ and $\Ind(\phi)=0$.
 
 Now suppose $l>0$, define the linear globally Lipschitz operators $\varphi_l,\phi_l: E \rightarrow F \times \rr^l$ by 
 $\varphi_l(x)\coloneq(\varphi(x),0)$ and $\phi_l(x)\coloneq(\phi(x),0)$.
 Then $\varphi_l$ is a Lipschitz-Fredholm operator of index 0. By the above argument $\phi_l$ is a Lipschitz-Fredholm
 operator of index 0 and hence $\phi$ is a Lipschitz-Fredholm operator of index $l$. Likewise, the case $l<0$ can be proved.
\end{proof}
\begin{theorem}\label{trans}
 Let $\varphi \in \mathcal{LF}_l(E,F)$, and suppose that $F_0 \subseteq F$ is closed and $\varphi \tn F_0$.
 Then any $\phi \in \mathcal{L}_{g,d}(E,F)$ in a fine neighborhood of $\varphi$ is transversal to $F_0$.
\end{theorem}
\begin{proof}
 By Proposition~\ref{spl} there exist closed subsets $E_0 \subseteq E$, $F_1 \subseteq F$ and $E_1 \coloneq \varphi^{-1}(F_1)$ such that
 $F = F_0 \oplus F_1$, $E = E_0 \oplus E_1$ and $\varphi_1 \coloneq \varphi |_{E_1} \in \Iso (E_1,F_1)$.
  There is a continuous function $\delta (x)$ such that every linear globally Lipschitz map $\psi: E_1 \rightarrow F_1$
 which is $\delta$-approximation to $\varphi_1$ is an isomorphism (see Remark~\ref{open2}).
 Let $\pi : F \rightarrow F_1$ be the projection given by $\pi(f_0+f_1)=f_1$, and let $\kappa = \mathrm{Id}_F - \pi$.
 It is immediate that $\pi$ is linear and globally Lipschitz and $\Img(\kappa) = F_0$. Choose $\varepsilon \in (0,\delta / \mathpzc{Lip}(\pi))$
 small enough, in view of Proposition~\ref{zero}, we can assume that each $\phi \in \mathcal{L}(E,F)$ which is $\varepsilon$-approximation to $\varphi$ 
belongs to $\mathcal{LF}_l(E,F)$. 

Now we show that each such $\phi$ is transversal to $F_0$. Let $\Phi \coloneq (\pi \circ \varphi)|_{E_1} \in \mathcal{L}(E_1,F_1)$.
Then $d(\Phi,\varphi_1) \leqq \mathpzc{Lip}(\pi)\varepsilon < \delta$ and so $\Phi \in \Iso(E_1,F_1)$ (see Remark~\ref{open2}).
Thus, we only need to prove $F=\Img(\phi)+F_0$. Let $f \in F=F_0 \oplus F_1$ so $f=f_0+f_1$, where $f_i \in F_i (i=0,1)$.
We have $\Phi^{-1}(f_1)=e_1 \in E_1 \subseteq E$, $x=\phi(e_1) \in \Img(\phi)$, and $y = f_0 - \kappa (x) \in F_0$. Whence,
$x+y= \pi(x)+\kappa(x)+f_0-\kappa(x)=f_0+\Phi(e_1)=f_0+f_1 =f$, therefore $F=\Img(\phi)+F_0$.
\end{proof}
Now we prove the transversality theorem for $\mc$-Lipschitz-Fredholm maps. It is indeed
a consequence of the Sard's theorem for these maps~\cite[Theorem 4.3]{k}. A careful reading of the
proof of the theorem shows that the minor assumption of endowing manifolds with compatible metrics is superfluous
and the theorem remains valid for manifolds without compatible metrics. Thus,
the statement of the theorem is as follows:
\begin{theorem}[Sard's Theorem]\label{sard}
 Let $ M$ and $ N $ be $\mc$ manifolds, $k\geqq1$. If $ f: M \rightarrow N $ is an
$\mc$-Lipschitz-Fredholm map with $  k > \max \lbrace {\Ind f,0} \rbrace $. Then, the set of regular
values of $ f $ is residual in $ N $.
\end{theorem}
\begin{defn}\label{transver}
 Let $f:M\rightarrow N$ be a Lipschitz-Fredholm map and let $\imath: \ab \hookrightarrow N$ be an $MC^1$ embedding
 of a finite dimensional manifold $\mathcal{A}$.
 We say that $f$ is transversal to $\imath$ and write $f \tn \imath $ if 
  $\operatorname{D}f(x)(T_xM) + \operatorname{D}{\imath(y)}(T_y\ab) =T_{f(x)}N$, whenever $f(x)= \imath (y)$.
  It is also said that the submanifold $\mathbf{A} \coloneq \imath{(\ab)}$ is transversal to $f$.
\end{defn}

The following theorem is the analogous of the Smale transversality~\cite[Theorem 3.1]{smale}.
Its proof is just a slight modification of the argument of Smale.

\begin{theorem}
Let $M$ and $N$ be  $\mc$ manifolds modelled on spaces $(F,d)$ and $(E,g)$, respectively.
 Let $f: M \rightarrow N$ be an $\mc$-Lipschitz-Fredholm map and
 let $\imath: \ab \hookrightarrow N$ be an $MC^1$-embedding of a finite dimension manifold $\mathcal{A}$
 with $k> \max \{ \Ind f + \dim \mathcal{A} ,0 \}$. Then
 there exists an $MC^1$ fine approximation $\mathbf{g}$ of $\imath$ such that $\mathbf{g}$ is embedding and  $f \tn \mathbf{g}$.
 Furthermore, Suppose $S$ is a closed subset of $\ab$ and $f \tn \imath(S)$, then 
 $\mathbf{g}$ can be chosen so that $\imath = \mathbf{g}$ on $S$.  
\end{theorem}
\begin{proof}
 Since manifolds are second countable we only need to work in local coordinates.
 Assume that $y \in \ab$ and $n = \dim \imath(\ab)$. Since $\imath(\ab)$ is an embedded submanifold of finite type of $N$,
 we may find an open neighborhood $U \subset \rr^n$ about $y$, a chart about $\imath(y)$ and a splitting $E = \rr^n \oplus E_1$ such that 
 $\imath(y) = \imath(y,0)$ in the neighborhoods. Let $\pi_2:E \rightarrow E_1$ be the projection onto $E_1$. Let
 $\overline{V} \subset U$ be a neighborhood of $y$, and $h$ a smooth
 real valued function which is $1$ on $\overline{V}$ and
 $0$ outside of $U$. Since $\pi_1 \circ f$ is locally Fredholm-Lipschitz
 it follows by Sard's Theorem (Theorem~\ref{sard}) that there is a regular value $z$ for $\pi_1 \circ f$ which is close to $0$.
 Now define
$$
 \mathbf{g}(y)= h(y)(z,y)+ [1-h(y)]\imath(y).
$$
It is immediate that $f \tn \mathbf{g}$ on $V$, and for $z$ sufficiently close to $0$, 
 $\mathbf{g}$ is $MC^1$ fine approximation to $\imath$. The second statement follows by our definition of $\mathbf{g}$.
 \end{proof}

\section{Transversal submanifolds}

We will need the following inverse function theorem.
\begin{theorem}[\cite{m}, Theorem 4.7. Inverse Function Theorem for $ MC^k$-maps] \label{invr}
Let $ (E,g) $ be a Fr\'{e}chet space with standard metric $g$. Let $U \subset E$ be open, $x_0 \in U$ and
$ f : U \subset E \rightarrow E $  an $ MC^k$-map, $ k \geq 1 $.
 If $f'(x_0) \in \Aut{(E)}  $, then there exists an open neighborhood
 $ V \subseteq U $ of $ x_0 $ such that $ f(V) $ is open in $ E $ and $ f\vert_V : V \rightarrow f(V) $ 
is an $ MC^k$- diffeomorphism.
\end{theorem}
To avoid some technical complications we consider only manifolds without  boundary in the sequel.
\begin{theorem}
 Let $M$ and $N$ be $\mc$ manifolds modelled on spaces $(F,d)$ and $(E,g)$, respectively. 
 Suppose that $f: M \rightarrow N$ is an $\mc$-Lipschitz-Fredholm map of index $l$. 
 Let $\ab$ be a submanifold of $N$ with dimension $m$ and let $\imath: \ab \hookrightarrow N$ be the inclusion.
 If $f$ is transversal to $\ab$, then $f^{-1}(\ab)$ is a submanifold of $M$ of dimension $l+m$. For all $x \in f^{-1}(\ab)$
 we have $T_x(f^{-1}(\ab))=(T_xf)^{-1}(T_{f(x)}\ab)$.
\end{theorem}
\begin{proof}
  If $f^{-1}(\ab) = \emptyset$ the theorem is clearly valid so let $f^{-1}(\ab) \neq \emptyset$.
 Let $(\psi,U)$ be a chart at $f(x_0) \in \ab$ in $N$ with the submanifold property for $\ab$.
 Let $U_1,U_2$ be open subsets of $E,\rr^m$ such that $\psi (U) = U_1 \oplus U_2, \psi(U \cap \ab) = U_1 \oplus \{ 0\}$, and 
 $\psi(f(x_0)) =(0,0)$. Let $(V,\varphi)$ be a chart at $x_0$ in $M$ such that $\varphi(x_0) =0$, 
 $\varphi :V \rightarrow \varphi(V) \subset F$ and $f(V) \subset U$. Let
 $$\mathbf{f} \coloneq  \psi \circ f\circ \varphi^{-1} : \varphi(V) \rightarrow \psi(U)$$
 be the local representative of $f$. Then $\mathbf{f}(0)= (0,0)$ and by hypothesis $\mathbf{f}$ is a  Lipschitz-Fredholm map, in particular, 
 $\operatorname{D}\mathbf{f}(0) \in \mathcal{LF}_l(F,E)$. The tangent map $T_{f(x_0)}\imath: T_{f(x_0)}\ab \rightarrow T_{f(x_0)}N$
 is injective with  closed split image. Hence $T_{f(x_0)}\ab$ can be identified with a closed split subspace of $T_{f(x_0)}N$.
 Thus $\operatorname{D}f(x_0)$ is transversal to $T_{f(x_0)}\ab$. Therefore, keeping in the mind the definition of the differential in terms of tangent maps,
 $\operatorname{D}\mathbf{f}(0)$ is transversal to $T\psi(T_{f(x_0)}\ab)= U_1 \oplus \{ 0\} \eqqcolon E_1$. Then, by virtue of
 Proposition~\ref{spl} there are closed subsets $F_1 \subset F$, $E_0 \subset E$ such that
 $F = F_1 \oplus (F_0 \coloneq \operatorname{D}\mathbf{f}(0)(E_0)) $, $E =E_1 \oplus E_0$, 
 $\Delta \coloneq \operatorname{D}\mathbf{f}(0) \mid_{F_0} \in \Iso(F_0,E_0) $  and 
  $\Delta_1 \coloneq \operatorname{D}\mathbf{f}(0) \mid_{F_1} \in \Iso(F_1,E_1) $.
  Moreover, $\dim F_0 = m+l$.
  
  Consider the projection $\pi : F \rightarrow F_1$ given by $\pi(f_0+f_1) =f_1$. Since $F_1$ and $F_0$ are closed and
  complementary it follows that obviously
  the map $\kappa = \mathrm{Id}_{F}-\pi$ is the unique projection with $\Img {(\kappa)} = F_0 $ and $\ker(\kappa)=F_1$ .
  Let $\pi_1 : E \rightarrow E_0 $ be the projection given by $\pi_1(e_0+e_1)=e_0$. Then,
  $\Pi \coloneq \Delta ^{-1} \circ \pi_1 \circ \operatorname{D}\mathbf{f}(0)$ is a projection 
  with $\Img{(\Pi)} = F_0$ and $F_1 \subseteq \ker{(\Pi)}$. Since $F=F_0 \oplus F_1$, it follows that
  $F_1 = \ker{(\Pi)}$ and therefore $\Pi= \kappa$.
  
  Now define the map $H: \varphi(V) \rightarrow F$ of class $\mc$ by
  $H(x) = \pi(x) + \Delta ^{-1} \circ \pi_1 \circ \mathbf{f}(x) $.
  We obtain that $H(0)=0$ and
  $\operatorname{D}H(0) = \pi + \Delta ^{-1} \circ \pi_1 \circ \operatorname{D}\mathbf{f}(0)= \pi + \kappa = \mathrm{Id}_F$.
  If we choose $V$ small enough, then by Theorem~\ref{invr} $H$ is an $\mc$-diffeomorphism onto an open neighborhood
  $\mathcal{U} \subseteq \psi(U)$ of $\psi(f(x_0) = (0,0)$. Let $\varPhi = H \circ \varphi^{-1}$, then
  $(\varPhi,F_0)$ is a chart for $x_0$ on $V$ with the submanifold property. Because we have
  $$
  x \in f^{-1}(\ab) \Leftrightarrow \psi(f(x)) \in U_1 \oplus \{0\} \Leftrightarrow \mathbf{f}(\varphi(x))
  \in U_1 \oplus \{0\} \Leftrightarrow H(\varphi(x)) \in F_0.
  $$
  Let $p \in \ab$, $\gamma : \rr \rightarrow M$ a smooth curve sending zero to $p$, and $j_p^1 \gamma$
  the $1$-jet of $\gamma$ at $p$.
  \begin{align}
   j^1_p \gamma \in T_p \ab \Leftrightarrow& j^1_{\varphi(p)}(\varphi \circ \gamma) = T\varphi(j^1_p\gamma) \in \varphi(V)\times F, \,  \varphi\circ \gamma \subset \varphi(V) \nonumber \\
   \Leftrightarrow& T \mathbf{f} (j^1_{\varphi(p)}(\varphi \circ \gamma)) \in \psi(U) \times E\nonumber \\
   \Leftrightarrow&\mathbf{f}(\varphi \circ \gamma) = \psi(f \circ \gamma) \subset \psi(U) \nonumber \\
   \Leftrightarrow& \dfrac{\mathrm{d}}{\mathrm{d}t} \psi(f \circ \gamma) \mid_{t=0} =v=
   \dfrac{\mathrm{d}}{\mathrm{d}t} \psi([\psi^{-1}(\psi (f(x))+tv)]), \psi( f \circ \gamma ) \subset \psi (U) \nonumber\\
   \Leftrightarrow& j^1_{f(p)}(f \circ \gamma)=j^1_p [\psi^{-1}(\psi (f(p))+tv)] \in T_{f(p)}\ab \nonumber \\
   \Leftrightarrow& j^1_p \gamma \in (T_pf)^{-1}(T_{f(p)}\ab) \nonumber
  \end{align}
This proves the second assertion.
\end{proof}
If  manifolds have nonempty boundary we just need to modify the proof by extending the considered maps.
\begin{cor}
 Let $f: M \rightarrow N$ be an $\mc$-Lipschitz-Fredholm map of index $l$. If $y$ is a regular value of  $f$,
 then the level set $f^{-1}(y)$ is a submanifold  of dimension $l$ and its tangent space at $x$ is $\ker T_xf$.
\end{cor}
\begin{proof}
 The set $\{y\}$ is transversal to $f$ so the result follows from the theorem.
\end{proof}
\begin{cor}
 Let $f : M \times N \rightarrow O$ be a smooth Lipschitz-Fredholm map of manifolds, we write $f_x \coloneq f(\cdot,x)$,  and let $\ab$  be a closed finite dimension submanifold
 of $O$. Assume that $f \tn \ab$ and for all $(m,n) \in f^{-1}_n(\ab)$ 
 the composition $(T_mM \xrightarrow{\operatorname{D}f_n(m)} T_{f_n(m)} O \xrightarrow{Q} T_{f_n(m)}O/T_nS )$ is Lipschitz-Fredholm. Then there is a residual
 set of $n$ in $O$ for which the map $f_n: M \rightarrow O$ is transversal to $\ab$.
\end{cor}
\begin{proof}
 By hypothesis the kernel of $Q \circ \operatorname{D}f(x)$ is complemented for all $x \in f^{-1}(\ab)$.
 By the preceding  theorem $B\coloneq f^{-1}(\ab)$ is a Fr\'{e}chet submanifold. The map $f \mid_B$ is smooth Lipschitz-Fredholm,
 therefore by Sard's theorem there is a residual set of regular values of it in $O$. If $n \in N$ is a regular value of
 $f \mid_B$, then $f_n$ is transversal to $\ab$.
 
\end{proof}


\begin{thebibliography}{1}
\bibitem{k1}
K.~Eftekharinasab, Geometry of bounded manifolds, Rocky Mountain J.~Math., to appear
\bibitem{k}
K.~Eftekharinasab, Sard's theorem for mappings between Fr\'{e}chet manifolds, Ukrainian Math. J. \textbf{62} (2010), 1634-1641.
 \bibitem{glockner}
H.~Gl\"{o}ckner, Implicit functions from topological vector spaces in the presence of metric estimates, preprint, Arxiv:math/6612673, 2006.
\bibitem{ko}
G.~K\"{o}ethe, Topological vector spaces I, Springer-Verlag, New York Inc., 1969.  
\bibitem{m} 
O.~M\"{u}ller, A metric approach to Fr\'{e}chet geometry, Journal of Geometry and physics  \textbf{58} (2008), 1477-1500.
\bibitem{smale}
Smale S. An infinite dimensional version of Sard's theorem,Amr.J.Math. \textbf{4} (1965) 861-866.
\end{thebibliography}
\end{document}